\documentclass[10pt]{amsart}
\usepackage{amsfonts,amssymb,amscd,amsmath,latexsym,enumerate,verbatim,calc,
quotmark}

\def\NZQ{\mathbb}               

\def\ZZ{{\NZQ Z}}
\def\Z{{\NZQ Z}}

\def\Um{{\operatorname{Um}}}          
\def\E{{\operatorname{E}}}            
\def\EO{{\operatorname{EO}}}            
\def\ESp{{\operatorname{ESp}}}          
\def\GL{{\operatorname{GL}}}          
\def\K{{\operatorname{K}}}            
\def\M{{\operatorname{M}}}            
\def\O{{\operatorname{O}}}
\def\K_1O{{\operatorname{K_1O}}}
\def\MS{{\operatorname{MS}}}          
\def\MSE{{\operatorname{MSE}}}        
\def\SL{{\operatorname{SL}}}          
\def\SK{{\operatorname{SK}}}          
\def\Sp{{\operatorname{Sp}}}          
\def\Um{{\operatorname{Um}}}          
\def\vpmod{\!\!\!\pmod} 
            
\def\WMS{{\operatorname{WMS}}}     
\def\wms{{\operatorname{wms}}}       
\def\SK{{\operatorname{SK}}}

\newtheorem{theorem}{Theorem}[section]
\newtheorem{lemma}[theorem]{Lemma}
\newtheorem{corollary}[theorem]{Corollary}
\newtheorem{proposition}[theorem]{Proposition}
\newtheorem{remark}[theorem]{Remark}

\newtheorem{definition}[theorem]{Definition}

\newtheorem{question}[theorem]{Question}
\newtheorem*{acknowledgement}{Acknowledgement}

\newcommand{\inD}[1][\relax]{\def\argone{#1}\def\temprelax{\relax}
  \ifx\argone\temprelax\right.\else\,\middle|#1\right.{}\fi}

\begin{document}

\title{\small A nice group structure on the 
orbit space of unimodular rows-II}
\author{{\bf Anjan Gupta, Anuradha Garge, Ravi A. Rao}\\
Tata Institute of Fundamental Research, Mumbai {\small \&} \\ 
 Center for Basic Sciences, Mumbai {\small \&}
\\  Tata Institute of Fundamental Research, Mumbai }
\footnote{Correspondence author: Ravi A. Rao; 
{\it email: ravi@math.tifr.res.in}}

\maketitle


\vskip3mm


\vskip3mm

\begin{abstract} We establish an Excision type theorem for niceness of
group structure on  the orbit space of unimodular rows of length $n$
modulo elementary action.  This permits us to establish niceness for
relative versions of results for the cases when $n = d+1$;  $d$ being
the dimension of the base algebra. We then study and establish
niceness for the case when $n = d$, and also  establish a relative
version, when the base ring is a smooth affine  algebra over an
algebraically closed field.
\end{abstract}

\section{Introduction}

In Algebraic Topology the Excision Theorem is a useful theorem about
relative homology: {\it viz.} given topological  spaces $X$ and
subspaces $A$ and $U$ such that $U$ is also a subspace of A, the
theorem says that under certain circumstances, we can cut out (excise)
$U$ from both spaces such that the relative homologies of the pairs
$(X,A)$ and $(X \setminus U,A \setminus U)$ are isomorphic. Succinctly, 
Excision preserves homology; but it is known that it does not preserve 
homotopy. 

Excision assists in computation  of singular homology groups, as sometimes
after excising an appropriately chosen subspace we obtain something
easier to compute. Or, in many cases, it allows the use of
induction. Coupled with the long exact sequence in homology, one can
derive another useful tool for the computation of homology groups, the
Mayer–Vietoris sequence. In the axiomatic approach to  homology, the
theorem is the sixth of the Eilenberg-Steenrod axioms (See \cite{es}).

In Algebra the above features of Excision were first introduced and studied by
Milnor \cite{Milnor} in his book on Algebraic $K$-theory. In this context, 
Milnor introduced the double of a ring $R \times_I R$ of a ring w.r.t. an 
ideal $I$. 

Later R.G. Swan studied in \cite{sw} whether Excision helped in
computing the  lower $K$-groups $K_1$, $K_2$; and showed that it
failed. 

The problem of characterizing the rings for which Excision holds was
very important from the very beginning of the development of algebraic
$K$-theory because of its relations to the Karoubi conjecture (on the equality 
of algebraic and topological  $K$-theory groups of stable
$C^{\ast}$-algebras), homology of congruence subgroups and other questions. In 
1992 Suslin and M. Wodzicki in \cite{sw1} solved the problem for rational 
algebraic $K$-theory. (Also see \cite{sw2}).

But prior to that, there are  two instances in Classical Algebraic 
$K$-theory where Suslin
uses Excision for the linear group--refer  (\cite{Mur-Gup}, 
Lemma $4.3$); and the  orthogonal group--refer (\cite{sk}, Corollary
$2.13$) (where they  prove that the relative orthogonal group
$\EO_{2r}(R, I)$ is a normal  subgroup of the orthogonal group
$\O_{2r}(R, I)$), and also to (\cite{sk},  Lemma $2.14$) where it is shown
that one can deduce the injective  stability for the relative
orthogonal quotients to $\K_1O(R, I)$ if one  knows it for the ring and
the double of the ring $R \times_I R$ w.r.t. the ideal $I$.

In \cite{vdK1} W. van der Kallen defined a group structure for the
orbits of unimodular rows of length $d + 1$, where $d$ was the
dimension  of the base ring, and studied the Excision property for
orbit  spaces $\MSE_n(R, I)$ of unimodular rows of length 
$n\geq 3$ modulo elementary action. (See  Theorem \ref{excision}). Later in 
\cite{vdK2} he showed that these  orbit spaces also have a group structure when
the size is a bit beyond half  the dimension (the so-called Borsuk
estimate). 

In \S 3 we deduce a Double Excision theorem, which is  a consequence
of his theorem; but simplifies its usage. Using this we  deduce in
Theorem  \ref{nice} that if $I$ is an ideal in a ring $R$, and the
orbit spaces  $\MSE_n(R)$, and $\MSE_n(R, I)$ have the usual group
structures (see  \cite{vdK1, vdK2}), then the group structure on
$\MSE_n(R, I)$ is nice  (i.e. is Mennicke-like) if it is nice for the
Excision ring  $\MSE_n(R \oplus I)$. (We call this relative niceness
criterion). (It would be interesting to know the appropriate analogue of the 
Double Excision theorem in Algebraic Topology.)

In \cite{Garge-Rao} group structure on the orbit space was shown to be
nice in the following cases:

\begin{itemize}
\item[{1.}] Let $A$ be an affine algebra (of dimension $d \geq 2$)
over a perfect field $k$, where ${\rm char}~k \neq 2$ and  the
cohomological dimension ${{\rm c.d.}_2}~k \leq 1$.  Then the group
structure on the orbit space  $\MSE_{d +1}(A)$ is
nice. 

\item[{2.}] Let $R$ be a commutative noetherian local ring  of
dimension $d \geq 3$, in which $2R = R$. Then the group structure on
$\MSE_{d +1}(R[X])$ is nice. 
\end{itemize}

We deduce from Double Excision that a relative version of the above
results also hold.   The key new observation here is Lemma \ref{local}
which asserts that if $R$ is a local ring then the Excision ring  $R
\oplus I$ is also a local ring.

We then begin the study of niceness for the orbit spaces $\Um_d(A, I)$,
when  $A$ is an affine algebra of dimension $d$ over an algebraically 
closed field $k$. The key new input which allows us to study this case
is the beautiful theorem of J. Fasel in (\cite{Fasel1}, Lemma 3.3) that 
for a smooth affine surface over an algebraically closed field of 
characteristic  $\neq 2, 3$, a stably elementary $2 \times 2$ 
matrix is stably elementary symplectic. Another useful observation used 
is in \cite{fsr2} which asserts that if $A$ is an affine threefold over an 
algebraically closed field then $\Um_4(A, (a)) = e_1\Sp_4(A, (a))$, for 
$a \in A$. 

If $A$ is smooth,  then we prove that the group  structure on
$\Um_d(A)/\E_{d}(A)$ is nice, when $k$ is algebraically closed,  and
 of characteristic different from  $2, 3$. 

Since $A \oplus I$ need not be smooth, even if $A$ is smooth, we are
unable to apply
the relative niceness criterion here. However, we are able  to circumvent
this, and deduce the relative version, under the above  assumptions on
the algebra $A$, and the assumption that $I$ is a principal ideal.

One of the interesting by-products of this paper is to get a relative
Mennicke-Newman Lemma (see (\cite{vdK3}, Lemma $3.2$) for the absolute
case,  and (\cite{vdK1}, Lemma $3.4$) for the relative case when
dealing with rows  of length $d + 1$, where $d$ is the dimension of
the base ring for the known cases earlier due to W.van der Kallen). It 
is  the use of this version of the Mennicke--Newman
lemma which permits us to  study the concept of niceness for rows of
smaller length and also to  realize that the concept of niceness does
not depend on `which coordinate' in  the relative case. 

In all the cases we have shown the niceness of $\MSE_n(A, I)$ it is 
known that the stably free projective $A$-modules of rank $n -1$ are 
free -- see \cite{Sus1}, \cite{fsr1}, \cite{fsr2} \cite{bq3}; in fact, some 
essential ingredient in proving the freeness has 
been used by us to prove the niceness. In  (\cite{Fasel1}, Theorem 2.1),  
J. Fasel has shown that when  $A$ is a smooth affine 
algebra of dimension $d \geq 3$ over a perfect field $k$ with c.d.$_2(k) 
\leq 2$ then $\WMS_{d+1}(A) = \MS_{d+1}(A)$; but by van der Kallen's 
theorem in \cite{vdK2} $\MSE_{d+1}(A) = \WMS_{d+1}(A)$; whence $\MSE_{d+1}(A)$ 
is nice. The result catches our attention as it is not known whether 
stably free projective $A$-modules of rank $d$ are free for such affine 
algebras $A$ when c.d.$_2k = 2$. We shall say more about this example in 
a sequel article.

\section{Preliminaries} 

Throughout this note, $R$ stands for a commutative ring with unity,
for $n \geq 1$, $\M_n(R)$ the set of all $n \times n$ matrices over
$R$ and  $\GL_n(R)$ the group of invertible $n \times n$ matrices over
$R$.  A row $v = (a_1, a_2, \ldots, a_n) \in R^{n}$ is said to be
unimodular of length $n$, if there is a row $w = (b_1, b_2, \ldots,
b_n) \in R^{n}$ such that  $\langle v, w \rangle := v \cdot w^{t} =
1$, where $w^{t}$ stands for the transpose of $w$. The set of all
unimodular rows of length $n$ over $R$ will be denoted by $\Um_n(R)$.
Given an ideal $I$ of a ring $R$, let $\Um_n(R, I)$ denote the subset
of $\Um_n(R)$ consisting of  unimodular rows $v = (a_1, a_2, \ldots,
a_n)$ with $v \equiv (1, 0, \ldots, 0) \vpmod I$ i.e.,  $v$ is
unimodular and $(a_1 - 1), a_2, \ldots, a_n \in I$. It can be shown
that for any $v \in \Um_n(R, I)$ there exists $w \in \Um_n(R, I)$ such
that $v \cdot w^t = 1$.  Given $\lambda \in R$, for $i \neq j$, let
$E_{ij}(\lambda) = I_n + \lambda{e_{ij}}$, where $I_n$ denotes the
identity matrix and $e_{ij} \in \M_n(R)$ is the matrix  whose only
non-zero entry is $1$ at the $(i,j)$-th position.  Such
$E_{ij}(\lambda)$'s are called elementary matrices.  The subgroup of
$\GL_n(R)$ generated by  $E_{ij}(\lambda), i \not= j, \lambda \in R$ is called the elementary
subgroup of $\GL_n(R)$ and will be denoted  by $\E_n(R)$. Similarly we define $\E_n(I)$ for any ideal $I$ in $R$. We now
recall the definition of the relative elementary group: 

\begin{definition} Let $I$ be an ideal of $R$. Then $\E_n(R, I)$ is
defined to be the smallest  normal subgroup of $\E_n(R)$ containing
the element $E_{21}(x), x \in I$. 
\end{definition}

For $n \geq 3$, the relative elementary group $\E_n(R, I)$ acts on the
set of relative unimodular  rows $\Um_{n}(R, I)$ and the orbit space
of relative unimodular rows under relative elementary action is
denoted by $\Um_n(R,I)/\E_n(R, I)$. We shall also use $\MSE_n(R, I)$
to denote the orbit space $\Um_n(R,I)/\E_n(R, I)$, following
\cite{vdK2}. (When $I = R$, this is the orbit space
$\Um_n(R)/\E_n(R)$.)  Following is due to H.Bass.

\begin{definition}\label{s range condition}{$($Stable range condition
$Sr_n(I))$} Let $I$ be an ideal in $R$. We shall say stable range
condition $Sr_n(I)$ holds for $I$ if for any $(a_1, a_2, \ldots,
a_{n+1})$ in $\Um_{n+1}(R, I)$ there exists $c_i$ in $I$ such that
$(a_1+c_1a_{n+1}, a_2 + c_2a_{n+1}, \ldots, a_n + c_na_{n+1}) \in
\Um_{n}(R, I)$.
\end{definition}

    We recall the following argument of Vaserstein (see \cite{Vas6})
for an ideal $I$ in $R$. Assume $Sr_n(I)$ holds for $I$. Let $(a_1,
a_2, \ldots, a_{n+2}) \in \Um_{n+2}(R, I)$. Then there exists $(b_1,
b_2, \ldots , b_{n+2}) \in Um_{n +2}(R, I)$  such that  $\Sigma_{i =
1}^{i = n + 2}a_ib_i = 1$. So $( a_1, a_2, \ldots, a_n ,
a_{n+1}b_{n+1} + a_{n+2}b_{n+2}) \in \Um_n(R, I)$. Now by the
condition $Sr_n(I)$ on $I$ we have $c_i\text{'s} \in I$ such that
$(a_i + c_i\{a_{n+1}b_{n+1} + a_{n+2}b_{n+2}\}) \in \Um_{n}(R, I)$. In
particular $(a_1 + c_1\{a_{n+1}b_{n+1} + a_{n+2}b_{n+2}\}, a_2 +
c_2\{a_{n+1}b_{n+1} + a_{n+2}b_{n+2}\}, \ldots , a_n +
c_n\{a_{n+1}b_{n+1} + a_{n+2}b_{n+2}\}, a_{n+1}) \in \Um_{n+1}(R,
I)$. Subtracting suitable multiples of $a_{n+1}$ from first $n$
coordinates we have $(a_1 + c_1 a_{n+2}b_{n+2}, a_2 +
c_2a_{n+2}b_{n+2}, \ldots , a_n + c_na_{n+2}b_{n+2}, a_{n+1}) \in
\Um_{n+1}(R, I)$. Therefore the condition $Sr_n(I)$ implies the
condition $Sr_{n+1}(I)$.

\begin{definition}\label{s range, s dim}{\rm (Stable range $Sr(I)$,
Stable dimension $Sd(I)$)} We shall define the stable range of $I$
denoted by $Sr(I)$ to be the least integer $n$ such that $Sr_n(I)$
holds for $I$. We shall define stable dimension of $I$ by  $Sd(I) =
Sr(I) - 1$.
\end{definition}

Following is proved in \cite{Vas6}.

\begin{lemma}\label{vaslemma}{\text{$(${\rm Vaserstein}$)$}} Let $I, J$ be
two ideals in $R$ such that $I \subset J$. Then the following are
true.
\begin{enumerate}[(a)]
\item $Sr(I) \leq Sr(J)$.
\item $Sr(J/I) \leq Sr(J)$.
\end{enumerate} In particular we have $Sr(I) \leq Sr(R)$ and $Sr(R/I)
\leq Sr(R)$. The above assertions are also true for stable dimension.
\end{lemma}

\begin{definition}\label{ering}$(${\rm Excision Ring}$)$
Let $R$ be a ring and $I$ an ideal in $R$. The Excision ring $\ZZ
\oplus I$, has coordinate-wise addition and  multiplication given by:
$(m, i) \cdot (n, j) = (mn, mj+ni+ij)$. The additive identity of this
ring is $(0, 0)$ and the  multiplicative identity is $(1, 0)$. We have
a ring homomorphism $f : \ZZ \oplus I \rightarrow R$ defined by $f(n,
i) = n + i$ which will induce a map $\mathfrak{f}:\Um_n(\ZZ \oplus I, 0 \oplus I ) \longrightarrow  \Um(R \oplus I, 0 \oplus I )$ defined by $(a_i) \mapsto (f(a_i))$.
\end{definition}
We recall Excision theorem (see \cite{vdK1}, Theorem
$3.21$): 

\begin{theorem}\label{excision}{\rm (Excision theorem)} Let $n \geq 3$
be an integer and $I$ be an ideal in a commutative ring $R$. Then the
natural maps  $F : \MSE_n(\ZZ \oplus I, 0 \oplus I) \rightarrow \MSE_n(R,
I)$ {\rm defined by} $[(a_i)] \mapsto [(f(a_i))]$ {\rm and}  $G : \MSE_n(\ZZ
\oplus I, 0 \oplus I) \rightarrow \MSE_n(\ZZ \oplus I)$ {\rm defined by} $ [(a_i)] \mapsto [(a_i)]$ are
bijections. 
\end{theorem}

\begin{definition}\label{vr}{{\rm(}Vaserstein, van der Kallen's rule
or $Vv$ rule{\rm)}} Let $v, w  \in Um_n(R)$, $n \geq 3$ be given by $ v
= (a, a_2, a_3, \ldots, a_{n})$, $w = (b, a_2, a_3, \ldots, a_{n})$. We
shall say that a group operation $\ast$ on $\MSE_n(R)$ is given by
Vaserstein, van der Kallen's rule  if it satisfies one of the
following equivalent conditions {\rm(}see remark following Theorem
$3.6$ in \cite{vdK1}{\rm)}.
\begin{itemize}
\item[1] Choose $p \in R$ such that  $ap \equiv 1\, \vpmod {(a_2, a_3,
\ldots a_n)}$. Then $$[w] \ast [v] = [(a(b+p) - 1, a_2(b+p), a_3, \ldots,
a_n)].$$

\item[2.]  Let $\alpha \in \M_2(R)$ such that $e_1\alpha = (a, a_2)$
and $ \overline{\alpha} \in \GL_2(\overline{R})$;  $\overline{R}=
R/(a_3, a_4, \ldots, a_n)$. Then $$[w] \ast [v] = [((b, a_2) \alpha , a_3,
\ldots, a_n)].$$
\end{itemize} A group operation $\ast$ on $\MSE_n(R, I)$ is said to
be given by Vaserstein, van der Kallen's rule if the induced operation
on $\MSE_n(\ZZ \oplus I)$ by $F, G$  (see Theorem \ref{excision}) follows
Vaserstein, van der Kallen's rule in the previous sense. We shall
abbreviate it as $Vv$ rule.

\end{definition}

\begin{definition}\label{weak symbol}{{\rm (Universal weak Mennicke
symbol $\WMS_n(R), n \geq 2$)}\,\rm ({\it cf.} \cite{vdK2}, Section $3$)} We
define the universal weak Mennicke symbol on $\MSE_{n}(R)$ by a set
map $\wms : \MSE_{n}(R) \longrightarrow \WMS_n(R)$, $[v] \longmapsto
\wms(v)$ to a group $\WMS_n(R)$. The group $\WMS_n(R)$ is the free
group generated by $\wms(v), v \in \Um_n(R)$ modulo the following
relations

\begin{enumerate}
\item $\wms(v) = \wms(vg)$ if $g \in \E_n(R)$.
\item If $(q,v_2, \ldots, v_n), (1+q,v_2, \ldots, v_n) \in \Um_n(R)$
and $r(1+q) = q \vpmod {(v_2, \ldots, v_n)}$, then  \\$\wms(q, v_2,
\ldots, v_n) = \wms(r, v_2, \ldots, v_n)\wms(1 + q, v_2, \ldots,
v_n)$.
\end{enumerate}
\end{definition}

In (\cite{vdK2}, Lemma $3.5$) van der Kallen has shown that if $[v], [w]
\in \MSE_{n}(R), n \geq 3, v = ( a, a_2,a_3 \ldots, a_{n})$, $w = ( b, a_2,a_3
\ldots, a_{n})$ and $p \in R$ such that $ap \equiv 1 \ \vpmod {(a_2,
a_3, \ldots a_n)}$, then 
$$\wms(w) \wms(v) = \wms((a(b+p) - 1, a_2(b+p), a_3, \ldots, a_n)).$$
We recall  (\cite{vdK2}, Theorem $4.1$).

\begin{theorem}\label{wms}{\rm(W. van der Kallen)} Let $R$ be a ring of stable dimension $d$,
$d \leq 2n -4$ and $n \geq 3$. Then the universal weak Mennicke symbol
$\wms : \MSE_{n}(R) \longrightarrow \WMS_n(R)$ is bijective and
$\WMS_n(R)$ has the structure of an abelian group given by $Vv$ rule.
\end{theorem}

\begin{remark}\label{vr} Let $I$ be an ideal in a ring $R$,  with
$max(R)$ a disjoint union of $V(I)$ and finitely many subsets $V_i$ each a noetherian topological space of dimension at most $d$. Then the maximal spectrum of $(\ZZ\oplus I)$ is the union of finitely many subspaces of dimension atmost $d$ whenever $d \geq 2$.  
Therefore $\ZZ \oplus I$ has stable dimension at most $d$ for $d \geq 2$ 
{\rm(see \cite{vdK1}, 3.19)}. So by the above Theorem $\MSE_n(\ZZ
\oplus I)$ has a group structure given by $Vv$ rule whenever  $n \geq
\text{max} \{3, 2 + \frac{d}{2}\}$. Equivalently $\MSE_n(R, I)$ has a
group structure given by $Vv$ rule whenever $n \geq  2
+ \frac{d}{2}, d \geq 2$. (The group structure on $MSE_n(R, I)$ is inherited 
from the one on the subgroup $MSE_n(\ZZ \oplus I, 0 \oplus I)$ of 
$MSE_n(\ZZ \oplus I)$).
\end{remark}

\begin{definition} We say that the group structure $\MSE_{n}(R)$, $n
\geq 3$, given by $Vv$ rule  is  {\bf nice}, if it is given by the
`coordinate-wise  multiplication' formula: 
$$[(b, a_2,  \ldots, a_{n})] \ast [(a, a_2, \ldots, a_{n})] = [(ab, a_2, \ldots, a_{n})].$$ 
A group structure on $\MSE_{n}(R, I)$ given by $Vv$ rule is said to be
{\bf nice} if it satisfies one of the following equivalent conditions
.

\begin{itemize}
\item[(1)]  $[(b, a_2,  \ldots, a_{n})]\ast [(a, a_2, \ldots, a_{n})] =
[(ab, a_2, \ldots, a_{n})]$.

\item[(2)]  $[(a_1,  \ldots, a_{n-1}, b)]\ast [(a_1, \ldots, a_{n-1}, a)] =
[(a_1, \ldots, a_{n-1}, ab)]$.

\item[(3)]  $[(a_1, \ldots, a_{i-1}, b, a_{i+1}, \ldots,
a_{n})]\ast [(a_1, \ldots, a_{i-1}, a, a_{i+1}, \ldots,  a_{n})] = [(a_1,
\ldots, a_{i-1}, ab, a_{i+1}, \ldots,  a_{n})]$ for $i = 2, 3, \ldots,
n$.

\end{itemize}
\end{definition}

\noindent {\it Proof of the equivalence:}\\ 
Given $ i \in I$ we  use $\tilde{i}$ to denote $(0, i) \in \ZZ \oplus I$. If $x = n + i, n \in \ZZ, i \in I$ then we define $\tilde{x} = (n, i) \in \ZZ \oplus I$ to be a preimage of $x$ under $f$ (see Definition \ref{ering}).
Equivalence of (2) and (3)
is obvious. So assume that (3) holds.  Let $(p, b_2, b_3, \ldots, b_n) \in
\Um_n(R, I)$ be such that $ap + \Sigma_{i =2}^na_ib_i = 1$. Now in
$\MSE_n(R, I)$ we have
\begin{align*} &[(b, a_2,  \ldots, a_{n})]\ast [(a, a_2, \ldots, a_{n})]
\\ & = [(a(b+p) - 1, a_2(b+p), a_3, \ldots, a_n)] \\ &= [(a(b+p) - 1,
a_2\lambda(b+p), a_3, \ldots, a_n)];\, \text{ assume $a(b+p) - 1 = 1 -
\lambda, \lambda \in I$}\\ &=[(a(b+p) - 1, a_2, a_3, \ldots,
a_n)] \ast [(a(b+p) - 1, \lambda(b+p), a_3, \ldots, a_n)]\\
&=[(a(b+p) - 1, a_2, a_3, \ldots, a_n)]\ast  [e_1]\\ 
&= [(ab, a_2, \ldots, a_{n})].
\end{align*}

Note that  $(\tilde{a}(\tilde{b}+\tilde{p}) - 1, \tilde{\lambda}(\tilde{b}+\tilde{p}), \tilde{a_3},
\ldots, \tilde{a_n})$ in $\Um_n(\ZZ \oplus I, 0 \oplus I)$ can be completed to a matrix in $\E_n(\ZZ
\oplus I)$ and therefore to a matrix in $\E_n(\ZZ \oplus I, 0 \oplus I)$ by Excision theorem \ref{excision}. So its image $(a(b+p) - 1, \lambda(b+p), a_3, \ldots, a_n) \in \Um_n(R, I)$ under $f$ can be completed to a matrix in $\E_n(R, I)$. This proves (1).

Now we assume that (1) holds. Then  $[(\tilde{a_1},  \ldots, \tilde{a_{n-1}},
\tilde{b})], [(\tilde{a_1}, \ldots, \tilde{a_{n-1}}, \tilde{a})] \in
\MSE_n( \ZZ \oplus I, 0 \oplus I) $ correspond to
$[(a_1,  \ldots, a_{n-1}, b)], [(a_1, \ldots, a_{n-1}, a)]$ by $F$
(see  Theorem  \ref{excision}) respectively. We choose  $(\tilde{b_1},
\tilde{b_2}, \ldots, \tilde {b_{n-1}}, \tilde{p}) \in \Um_n(\ZZ \oplus
I, 0 \oplus I)$ such that $\Sigma_{i=1}^{n-1}\tilde {a_i}\tilde{b_i} +
\tilde{a}\tilde{p} = 1$. Since the induced operation on $\MSE_n(\ZZ
\oplus I)$ satisfies $Vv$ rule we have 
\begin{eqnarray*} &&[(\tilde{a_1},  \ldots, \tilde{a_{n-1}},
\tilde{b})]\ast [(\tilde{a_1}, \ldots, \tilde{a_{n-1}}, \tilde{a})] \\
&&=[(\tilde{a_1}(\tilde{b}+\tilde{p}), \tilde{a_2}, \ldots,
\tilde{a_{n-1}}, \tilde{a}(\tilde{b}+\tilde{p}) -1)]\\
&&=[(\tilde{a_1}(\tilde{b}+\tilde{p} +
1-\tilde{a}(\tilde{b}+\tilde{p})), \tilde{a_2}, \ldots,
\tilde{a_{n-1}}, \tilde{a}(\tilde{b}+\tilde{p}) -1)]\\
&&=[(\tilde{a_1}(\tilde{b}+\tilde{p} +
1-\tilde{a}(\tilde{b}+\tilde{p})), \tilde{a_2}, \ldots,
\tilde{a_{n-1}}, \tilde{a}(\tilde{b}+\tilde{p}) -1 +
\tilde{a_1}(\tilde{b}+\tilde{p} + 1-\tilde{a}(\tilde{b}+\tilde{p})))]
\ldots \ldots(a).
\end{eqnarray*} Now both $(\tilde{a_1}, \tilde{a_2}, \ldots,
\tilde{a_{n-1}}, \tilde{a}(\tilde{b}+\tilde{p}) -1 +
\tilde{a_1}(\tilde{b}+\tilde{p} + 1-\tilde{a}(\tilde{b}+\tilde{p})))$
and $(\tilde{b}+\tilde{p} + 1-\tilde{a}(\tilde{b}+\tilde{p}),
\tilde{a_2}, \ldots, \tilde{a_{n-1}}, \tilde{a}(\tilde{b}+\tilde{p})
-1 + \tilde{a_1}(\tilde{b}+\tilde{p} +
1-\tilde{a}(\tilde{b}+\tilde{p})))$ are in $\MSE_n(\ZZ \oplus I,
0 \oplus I)$. Then by our given hypothesis $(1)$ equation  $(a)$ equals
\begin{eqnarray*} &&[(\tilde{a_1}, \tilde{a_2}, \ldots,
\tilde{a_{n-1}}, \tilde{a}(\tilde{b}+\tilde{p}) -1 +
\tilde{a_1}(\tilde{b}+\tilde{p} +
1-\tilde{a}(\tilde{b}+\tilde{p})))]\ast [(\tilde{b}+\tilde{p} +
1-\tilde{a}(\tilde{b}+\tilde{p}), \tilde{a_2}, \ldots,
\tilde{a_{n-1}}, \\ &&\tilde{a}(\tilde{b}+\tilde{p}) -1 +
\tilde{a_1}(\tilde{b}+\tilde{p} +
1-\tilde{a}(\tilde{b}+\tilde{p})))]\\ &&=[(\tilde{a_1}, \tilde{a_2},
\ldots, \tilde{a_{n-1}}, \tilde{a}(\tilde{b}+\tilde{p}) -1
)]\ast [(\tilde{b}+\tilde{p} + 1-\tilde{a}(\tilde{b}+\tilde{p}),
\tilde{a_2}, \ldots, \tilde{a_{n-1}}, \tilde{a}(\tilde{b}+\tilde{p})
-1 )]\\ &&= [(\tilde{a_1}, \tilde{a_2}, \ldots, \tilde{a_{n-1}},
\tilde{a}(\tilde{b}+\tilde{p}) -1)]\ast [(\tilde{b}+\tilde{p},
\tilde{a_2}, \ldots, \tilde{a_{n-1}}, \tilde{a}(\tilde{b}+\tilde{p})
-1 )]\\ &&=[(\tilde{a_1}, \tilde{a_2}, \ldots, \tilde{a_{n-1}},
\tilde{a}(\tilde{b}+\tilde{p}) -1)]\ast [e_1]\\ &&=[(\tilde{a_1},
\tilde{a_2}, \ldots, \tilde{a_{n-1}}, \tilde{a}\tilde{b})]; \,
\text{since $\tilde{a}\tilde{p} \equiv 1 \vpmod {(\tilde{a_1},
\tilde{a_2}, \ldots, \tilde{a_{n-1}})}$}  \ldots \ldots(b). 
\end{eqnarray*} 
Comparing both sides of the equation
$[(\tilde{a_1},  \ldots, \tilde{a_{n-1}}, \tilde{b})]\ast [(\tilde{a_1},
\ldots, \tilde{a_{n-1}}, \tilde{a})] = [(\tilde{a_1}, \tilde{a_2},
\ldots, \tilde{a_{n-1}}, \tilde{a}\tilde{b})]$ under $F$ in $\MSE_n(R,
I)$ we get (2).

\section{On van der Kallen's Excision theorem}  In this section we
recall the construction and properties of the Excision ring.  Let $R$
be a ring and $I$ an ideal in $R$. The Excision ring  
$R \oplus I$, has
coordinate-wise addition and  multiplication given by: $(r, i) \cdot
(s, j) = (rs, rj+si+ij)$. The additive identity of this ring is $(0,
0)$ and the  multiplicative identity is $(1, 0)$. We use the Excision
Theorem  to prove: 

\begin{lemma}\label{lem:bijection}{{\rm(}Double Excision{\rm)}}  Let
$R$ be a ring and $I$ an ideal in $R$. Under the natural maps, for $n
\geq 3$, the following orbit spaces are in bijection:
$$\MSE_n(R \oplus I, 0 \oplus I) \leftrightarrow \MSE_n(\ZZ \oplus (0 \oplus I), (0 \oplus I))
\leftrightarrow \MSE_n(\ZZ \oplus (0 \oplus I))$$ 
$$\leftrightarrow \MSE_n(\ZZ \oplus I) \leftrightarrow \MSE_n(\ZZ \oplus I, 0 \oplus I)
\leftrightarrow \MSE_n(R, I).$$

 Let $\pi_2: R \oplus I \rightarrow R$ be the surjective map given by
$\pi_2(a, i) = a+i$. Assume that for  $n \geq 3$, there exist group
structures (with product given by the van der Kallen formula)  on the
orbit spaces  $\MSE_{n}(R, I)$  and  $\MSE_{n}(R \oplus I, 0 \oplus I)$ 
then, $\pi_2: \MSE_{n}(R \oplus I, 0 \oplus I) \rightarrow \MSE_{n}(R,
I)$ is a group homomorphism.  In particular, $\pi_2$ preserves the
nice group structure. 
\end{lemma}

\begin{proof} That there is a bijection between the first three orbit
spaces listed above follows from Excision theorem \ref{excision}. Similarly, that there is a bijection
between the last three orbit spaces  also follows from Excision
theorem.  

It only remains to check that there is bijection between $\MSE_n(\Z
\oplus 0 \oplus I)$ and $\MSE_n(\Z \oplus I)$. But this follows from
the fact that $\varphi: \Z \oplus (0 \oplus I) \rightarrow \Z \oplus
I$ by $\varphi((m,0, i)) = (m, i)$ is a ring isomorphism inducing an
isomorphism in $\MSE_n$.

The last assertion regarding group homomorphism follows from the fact
that $\pi_2$ respects the ring structure  of the Excision ring. 
\end{proof} 

\begin{definition} We shall say a ring homomorphism $\phi: B
\longrightarrow D$ is a retract if there exists a ring homomorphism
$\gamma: D \longrightarrow B$ so that $\phi \circ \gamma $ is
identity on $D$. We shall also say that $D$ is a retract of $B$.
\end{definition} Note that if $\phi: B \longrightarrow D$ is a retract
then $\phi$ induces an onto map from $\Um_n(B)$ to $\Um_n(D)$. We recall
a Lemma of Suslin (see \cite{Mur-Gup}, Lemma $4.3$), which gives a
handle on the relative  elementary group in certain special cases.

\begin{lemma}\label{lem:august15}  Let $B, D$ be rings and let $D$ be
a retract of $B$ and let $\pi : B \twoheadrightarrow D$.  If $J = {\rm
ker}(\pi)$, then $\E_{n}(B, J) = \E_{n}(B) \cap \SL_{n}(B, J), n \geq 3$.
\end{lemma} 

We isolate here another result which is a consequence of Lemma
\ref{lem:august15} above  and which will be used repeatedly throughout
this paper. 

\begin{lemma}\label{l} Let the quotient map $q : R \longrightarrow
R/I$ be a retract. Let $v \in \Um_n(R, I)$ be such that its class
$[v]$ is trivial in $\MSE_n(R), n \geq 3$. Then $[v]$ is also trivial in
$\MSE_n(R, I)$.
\end{lemma}

\begin{proof} We have a ring homomorphism $f: R/ I \longrightarrow R$
such that $q \circ f = id$. By hypothesis there exists a $\varepsilon
\in \E_n(R)$ such that $v\varepsilon = e_1$. Taking the image in $R/I$
we have $e_1 q(\varepsilon) = e_1$ and therefore $e_1 f\circ
q(\varepsilon) = e_1$. Let $\varepsilon' = \varepsilon (f\circ
q(\varepsilon))^{-1} \in \E_n(R) \cap \SL_n(R, I)$. Then by Lemma
\ref{lem:august15} we have $\varepsilon' \in \E_n(R, I)$ and
$v\varepsilon' = e_1$ holds obviously. So $[v]$ is trivial in
$\MSE_n(R, I)$.
\end{proof}

A special case of the above lemma says the following.

\begin{corollary}\label{cor:repeat}  Let $R$ be a ring and $I$ be an
ideal in $R$ and $n \geq 3$ be an integer.  If $[v] \in \Um_{n}(R
\oplus I, 0 \oplus I)$ is such that $[v] = [e_1]$ in $\MSE_{n}(R
\oplus I)$, then  $[v] = [e_1]$ in $\MSE_{n}(R \oplus I, 0 \oplus I).$  
\end{corollary}

%

\begin{lemma}\label{nice} $($Relative Niceness Criterion$)$

Let $R \oplus I$ be the Excision ring 
of $R$ with respect to an ideal $I$ in $R$ and $n \geq 3$. Suppose both $\MSE_{n}(R,
I)$ and $\MSE_{n}(R \oplus I)$ have group structures  given by $Vv$
rule. Then the group structure on $\MSE_{n}(R, I)$ is nice whenever it
is nice for $\MSE_{n}(R \oplus I)$.
\end{lemma}

\begin{proof} Corollary \ref{cor:repeat} shows that the map $\phi :
\MSE_{n}(R \oplus I, 0 \oplus I) \longrightarrow \MSE_{n}(R \oplus I)$
sending the relative class of a unimodular row $v \in \Um_n(R \oplus
I, 0 \oplus I)$ to its absolute class is an injective group
homomorphism. So if the group structure on $\MSE_{n}(R \oplus I)$ is
nice then it is so on $\MSE_{n}(R \oplus I, 0 \oplus I)$ also. Now by
Double Excision Lemma \ref{lem:bijection} we have $\MSE_{n}(R \oplus
I, 0 \oplus I) = \MSE_{n}(R ,  I)$. So the result follows.
\end{proof}

By Theorem \ref{lem:bijection} we have $\MSE_{n}(R \oplus I, 0 \oplus
I) \cong \MSE_{n}(\ZZ \oplus I)$. So Corollary \ref{cor:repeat} leads
us to ask the following.

\begin{question}
 Is the map $\MSE_{n}(\ZZ \oplus I) \longrightarrow
\MSE_{n}(R \oplus I), n \geq 3$ injective?
\end{question}

Lemma \ref{nice} leads us to ask the following.
\begin{question}
Is it true that the group structure on $\MSE_n(\ZZ \oplus I, 0 \oplus I)$ is 
nice if and 
only if the group structure on $\MSE_n(\ZZ \oplus I)$ is nice, 
when both have a group structure given by $Vv$ rule?

\end{question}


\section{Relative orbit space: size $(d + 1)$}

It is well known that the double of a ring w.r.t. 
an ideal is the same as the Excision ring w.r.t. an ideal $I$. The reader 
may look at the reference below for details, if necessary. 

\begin{proposition}\label{prop:Keshari-affine}$(${\it cf.} {\cite{Keshari}\rm,  Proposition 3.1}$)$ 
Let $R$ be a ring of
dimension $d$ and $I$ a finitely generated ideal of $R$. 

Consider the Cartesian square: 
$$\begin{CD}
C @>>> R \\ @VVV  @VVV\\ R @>>> R/I 
\end{CD}$$  Then, $C$ is finitely generated algebra of dimension $d$
over $R$ and integral over $R$. In fact, $C \simeq R \oplus I$ with
coordinate wise addition and multiplication defined by $(a, i)(b,j) =
(ab, aj + ib +ij)$. 

In particular, if $R$ is an affine algebra of
dimension $d$ over a field $k$, then $C$ is also an affine algebra of
dimension $d$ over $k$. 
\end{proposition}

\begin{theorem}\label{nice affine} Let $A$ be an affine algebra of
dimension $d \geq 2$ over a perfect field $k$, with ${\rm char}~k \neq
2$ and  the cohomological dimension ${{\rm c.d.}_2}~k \leq 1$.  Let
$I$ be an ideal of $A$. Then the group structure on $\MSE_{d+1}(A, I)$
is nice.
\end{theorem}

\begin{proof} By Lemma \ref{nice} it is enough to prove that the group
structure on $\MSE_{d+1}(A \oplus I)$ is nice. Now $A \oplus I$ is an
affine algebra of dimension $d$ over $k$ by Proposition
\ref{prop:Keshari-affine}. So the result follows from
(\cite{Garge-Rao}, Theorem $3.9$).
\end{proof}

\begin{lemma}\label{local} Let $(R, \mathfrak{m})$ be a local ring
with maximal ideal $\mathfrak{m}$. Then the Excision ring $R \oplus
I$ with respect to a proper ideal $I$ in $R$ is also a local ring with
maximal ideal $\mathfrak{m} \oplus I$. 
\end{lemma}

\begin{proof} $R \oplus I$ is a commutative ring with identity
$(1,0)$. For any  $i \in I \subset \mathfrak{m}$, $1+i$ is a unit in
$R$ with inverse of the form $1 +j$ for some $j \in I$. Therefore
$(1,0)+ (0,i) = (1, i)$  is a unit in $R \oplus I$ with inverse $(1,
j)$. So $0 \oplus I$ is contained in the Jacobson radical of $R \oplus
I$. We also have $\mathfrak{m} \oplus 0$ contained in the Jacobson
radical since any element in $(1, 0) + \mathfrak{m} \oplus 0$ is a
unit in $R \oplus I$. So $\mathfrak{m} \oplus I$ is contained in the
Jacobson radical. But $\mathfrak{m} \oplus I$ is a maximal ideal in $R
\oplus I$. Hence the result follows.
\end{proof}

\begin{theorem}\label{nice extended} Let $(R, \mathfrak{m})$ be a
commutative, noetherian, local ring  of dimension $d \geq 2$, in which
$2R = R$. Let $I$ be a proper ideal in $R$.  Then the group structure
on $\MSE_{d+1}(R[X], I[X])$ is nice. 
\end{theorem} 

\begin{proof} By Lemma \ref{nice} it is enough to prove that the group
structure on $\MSE_{d+1}((R \oplus I)[X])$ is nice. But $R \oplus I$
is a local ring by Lemma \ref{local}. So the result follows from
Theorem $5.1$ in \cite{Garge-Rao}. 
\end{proof}

We recall (\cite{vdK2}, Theorem  $2.2$) of van der Kallen.
\begin{theorem}\label{v} $($W. van der Kallen$)$

Let $n \geq 3$. Assume that $R$ is
commutative with $Sd(R) \leq 2n - 3$ or assume that the maximal
spectrum of $R$ is the union of finitely many noetherian subspaces of
dimension at most $2n - 3$. Let $i,j$ be non-negative integers. For
every $\sigma \in \GL_{n + i}(R) \cap \E_{n+i+j+1}(R, I)$ there are
matrices $u, v, w, M$ with entries in $I$ and $q$ with entries in $R$
such that

$$\begin{pmatrix}
I_{i+1} + uq & v \\ wq           & I_{n-1} + M      
\end{pmatrix} \in \sigma \E_{n+i}(R, I), \quad
\begin{pmatrix} I_{j+1} + qu & qv \\ w            & I_{n-1} + M      

\end{pmatrix} \in \E_{n+j}(R, I).
$$
\end{theorem}

\begin{corollary}\label{cor:completable-affine-d}  Let $A$ be an
affine algebra of dimension $d \geq 2$ over an algebraically closed
perfect field $k$, with ${\rm char}~k \neq 2$ and  the cohomological
dimension ${{\rm c.d.}_2}~k \leq 1$. Let $\sigma \in \SL_{d+1}(A, I)
\cap \E_{d+2}(A, I)$. Then, $[e_1\sigma] = [e_1]$ in $\MSE_{d+1}(A,
I)$ i.e. $e_1\sigma$ is relatively elementarily equivalent to $e_1$.  
\end{corollary}

\begin{proof} Putting $i = j = 0$ and $n = d + 1 $ in the Theorem
\ref{v} we have
 
$$\begin{pmatrix}
1 + uq & v \\ wq           & I_{d} + M      
\end{pmatrix} \in \sigma \E_{d+1}(R, I), \quad
\begin{pmatrix} 1 + qu & qv \\ w            & I_{d} + M      

\end{pmatrix} \in \E_{d+1}(R, I).
$$
Therefore in $\MSE_{d+1}(R, I)$ we have
\begin{eqnarray*} [e_1\sigma]\\ &=& [(1 + uq, v)]\\ &=& [(1+uq, v)]
\ast [(1+uq, q)]; \quad \text{since $[(1+uq, q)]$ is the identity
$[e_1]$}\\ &=& [(1 + uq, qv)]; \quad \text{since the group structure
on $\MSE_{d+1}(R, I)$ is nice by Theorem \ref{nice affine}}\\ &=& [e_1].
\end{eqnarray*}
\end{proof}

Similarly using Theorem \ref{nice extended} we have the following.

\begin{corollary}\label{cor:completable-polynomial-(d+1)}  Let $(R,
\mathfrak{m})$ be a commutative, noetherian, local ring of dimension
$d \geq 2$, in which $2R = R$ and $\sigma(X) \in \SL_{d+1}(R[X], I[X])
\cap \E_{d+2}(R[X], I[X])$. Then, $e_1\sigma(X)$ is relatively
elementarily equivalent to $e_1$. 
\end{corollary}

\section{Improved injective stability in relative case}

In this section  we shall recall the following relative version of 
(\cite{vdK-Rao}, Theorem 3.4) with respect to a principal ideal.

%

\begin{lemma}\label{ris}
{\rm(}{\it c.f.} \cite{AG}$)$
Let $A$ be an affine algebra of dimension $d \geq 2$ over an algebraically closed field $k$ and $I = (a)$ a principal ideal in $A$. Let $\alpha \in \SL_{d+1}(A, I)\cap \E(A, I)$. Then $\alpha$ is isotopic to identity relative to $I$. Moreover if $A$ is nonsingular then, 
\[\SL_{d+1}(A, I) \cap \E(A, I) = \E_{d+1}(A, I).\]
\end{lemma}

%

\begin{lemma}\label{a} Let $R$ be a commutative ring and $I$ an ideal
in $R$. Let $u, v \in \Um_3(R, I)$ such that $u \alpha = v$ for some
$\alpha \in \SL_3(R, I) \cap \E_4(R, I)$. Then $u$ and $v$ are
elementary equivalent relative to $I$.
\end{lemma}

\begin{proof}
 Let $\E^k_n(R, I)$ be the subgroup of $\GL_n(R)$ generated by the $E_{ki}(a)$ with $a \in R, i \not= k$ and the $E_{ik}(x), x \in I, i \not= k$. In (\cite{vdK1}, Lemma 2.2) it has been shown that $\E_n(R, I) = \E^1_n(R, I) \cap \GL_n(R, I)$. Let $\sigma$ be the permutation matrix obtained by interchanging the first and $k$th row of $I_n$. Then $$\E_n(R, I) = \sigma \E_n(R, I) \sigma^{-1} = \sigma \E^1_n(R, I) \sigma^{-1} \cap \GL_n(R, I) = \E^k_n(R, I) \cap \GL_n(R, I).$$ In particular any matrix in $\E_4(R, I)$ can be expressed as product of elementary matrices of the form $E_{4i}(a); a \in R, 1 \leq i \leq 3$ and $E_{i4}(x); x \in I, 1 \leq i \leq 3$. 

We shall show that $u \alpha \in u\E_3(R, I)$. Let $u = (u_1, u_2, u_3)$ and $w = (w_1, w_2, w_3) \in \Um_3(R, I)$ be such that $u_1w_1 + u_2w_2 + u_3w_3 = 1$. Define $$\theta(w, u) = \begin{pmatrix}0 & -u_1 & -u_2 & -u_3\\ u_1 & 0 & -w_3 & w_2 \\ u_2 & w_3 & 0 & -w_1 \\ u_3 & -w_2 & w_1 & 0   \end{pmatrix}.$$ 

We have $(\begin{smallmatrix} 1 &0 \\ 0 & \alpha \end{smallmatrix}) \in \E_4(R, I)$ and $(\begin{smallmatrix} 1 &0 \\ 0 & \alpha \end{smallmatrix})^t\theta(w, u)(\begin{smallmatrix} 1 &0 \\ 0 & \alpha \end{smallmatrix}) = \theta(w', u \alpha)$, for some $w' \in \Um_3(R, I)$. Now $(\begin{smallmatrix} 1 &0 \\ 0 & \alpha \end{smallmatrix})^t \in \E_4(R, I)$ and therefore is product of elementary matrices of the form $E_{4i}(a); a \in R, 1 \leq i \leq 3$ and $E_{i4}(x); x \in I, 1 \leq i \leq 3$. Let $\beta \theta(w, u)\beta^t = \theta(w', u\hat{\beta}) $ for such elementary matrix $\beta$. We shall compute $\hat{\beta}$ when $\beta$ is elementary matrices of different type as described above.
So first assume $\beta = E_{14}(x), x \in I$. Then $$(u_1 + xw_2, u_2 - xw_1, u_3) = u \begin{pmatrix}1 + xw_1w_2 & -xw_1^2 & 0 \\ xw_2^2 & 1 - xw_1w_2 & 0 \\ x w_2w_3 & -xw_1w_3 & 1 \end{pmatrix} \in u\begin{pmatrix}I_2 +  \nu \mu & 0 \\ I \times I & 1 \end{pmatrix}$$ for $\nu = (w_1, w_2)^t$ and $\mu = (xw_2,  -xw_1)$. Note that $\mu \nu = 0$. So we choose $\hat{\beta} = \begin{pmatrix}I_2 +  \nu \mu & 0 \\ I \times I & 1 \end{pmatrix} \in \E_3(A, I)$ by (\cite{Vas5}, Lemma 1.1(b)). In the other cases finding $\hat{\beta}$ is easy. We have 
\begin{eqnarray*}
\hat{\beta} = 
\begin{cases}
I_3 & \text{when $\beta = E_{41}(a), a \in R$}\\
\gamma^t & \text{when $\beta = E_{4i}(a), i = 2, 3, a \in R$ or $E_{i4}(x), i = 2, 3, x \in I$ {\it i.e.} $\beta $ is of the form $(\begin{smallmatrix}1 & 0\\ 0 & \gamma \end{smallmatrix})$.}
\end{cases}
\end{eqnarray*}

 So we have $\theta(w', u \alpha) = (\begin{smallmatrix} 1 &0 \\ 0 & \alpha \end{smallmatrix})^t\theta(w, u)(\begin{smallmatrix} 1 &0 \\ 0 & \alpha \end{smallmatrix}) = \theta(w', u \hat{\alpha})$ where $\hat{\alpha} = \prod \hat{\beta}$. Clearly $\hat{\alpha}^t \in \E^3_3(R, I)$. It is easy to see that $\hat{\alpha}^t \in \GL_3(R, I)$ since $(\begin{smallmatrix} 1 &0 \\ 0 & \alpha \end{smallmatrix})^t = \prod \beta \in  \GL_3(R, I)$. So $\hat{\alpha}^t \in  \E_3(R, I)$ and therefore $\hat{\alpha} \in  \E_3(R, I)$. Thus $v = u \alpha = u \hat{\alpha} \in u \E_3(R, I)$.
\end{proof}

\begin{theorem}\label{witt relative} Let $R$ be any commutative ring of
dimension $3$ and $I$ an ideal in $R$ such that $\SL_4(R,I) \cap \E(R,
I) = \E_4(R, I)$. Then $\MSE_3(R, I)$ has an abelian Witt group structure
given by $Vv$ rule.
\end{theorem}
\begin{proof} Since $R$ has dimension $3$ and $\SL_4(R,I) \cap \E(R,
I) = \E_4(R, I)$ the natural map $$\SL_n(R,I)/ \E_n(R, I) \rightarrow
\SK_1(R, I) $$ is a bijection whenever $n \geq 4$. Now the maximal
spectrum of the Excision algebra $\ZZ \oplus I$ is the union of
finitely many subspaces of dimension at most $3$ (See \cite{vdK1},
3.19). So we have \[e_1\SL_{2r +1}(\ZZ \oplus I) = \Um_{2r+1}(\ZZ
\oplus I) \,\text{for all}\, r \geq 2.\]

Now assume $v \in \Um_{2r}(\ZZ \oplus I)$, $r \geq 2$ which is stably
elementary equivalent to $e_1$. By elementary operations if necessary
we may assume that $v = e_1 \, \vpmod I$. Thus $v \alpha = e_1$ for some $\alpha \in \SL_{2r}(\ZZ \oplus I) \cap \E(\ZZ \oplus I)$. Going modulo $0 \oplus I$ we have $e_1\overline{\alpha} = e_1$ for $\overline{\alpha} \in \E_{2r}(\ZZ)$. Replacing $\alpha$ by $\alpha \overline{\alpha}^{-1}$ we may assume that $\alpha \in \SL_{2r}(\ZZ \oplus I, 0\oplus I) \cap \E(\ZZ \oplus I, 0 \oplus I)$ (see Lemma \ref{lem:august15}).

Let $\widetilde{v}, \widetilde{\alpha}$ be the
image in $\Um_{2r}(R, I), \SL_{2r}(R, I)$ respectively under the maps induced by $f$ (see Definition \ref{ering}). Then $\widetilde{v}\widetilde{\alpha} = e_1 $, $\widetilde{\alpha} \in \SL_{2r}(R, I) \cap \E(R, I)$. So  $\widetilde{\alpha} \in \E_{2r}(R, I)$ by given hypothesis and   
$\widetilde{v}$ is trivial in $\MSE_{2r}(R,
I)$. Then by Excision Theorem \ref{excision}, $v$ is also trivial in
$\MSE_{2r}(\ZZ \oplus I)$ i.e. $v \in e_1 \E_{2r}(\ZZ \oplus I)$. Thus
we have \[e_1(SL_{2r}(\ZZ \oplus I) \cap \E(\ZZ \oplus I)) = e_1
\E_{2r}(\ZZ \oplus I),\, \text{whenever}\, r \geq 2.\]

Now (\cite{Sus-Vas}, Theorem 5.2 b, c) shows that $\Um_3(\ZZ \oplus
I)/ \SL_3(\ZZ \oplus I) \cap \E(\ZZ \oplus I) = W_E(\ZZ \oplus I)
$. We claim that $\SL_3(\ZZ \oplus I) \cap \E(\ZZ \oplus I)$ and
elementary orbits are same in $\Um_3(\ZZ \oplus I)$. Choose $u, v \in
\Um_3(\ZZ \oplus I)$ such that $u \alpha = v$ for some $\alpha \in
\SL_3(\ZZ \oplus I) \cap \E(\ZZ \oplus I)$. By elementary action ({\it viz.} $\E_3(\ZZ)$ action) if necessary we may assume that  $u, v \in
\Um_3(\ZZ \oplus I, 0 \oplus I)$ and $\alpha \in \SL_3(\ZZ \oplus I, 0 \oplus I) \cap \E(\ZZ \oplus I, 0 \oplus I) $. Taking images in $\Um_3(R,
I)$ as earlier we have $\widetilde{u} \widetilde{\alpha} =
\widetilde{v}$ where $\widetilde{\alpha} = \SL_3(R, I) \cap \E(R, I) =
\SL_3(R, I) \cap \E_4(R, I)$. Then by Lemma \ref{a}, $[\widetilde{v}] =
[\widetilde{u}]$ in $\MSE_3(R, I)$ and therefore $[v] = [u]$ in
$\MSE_3(\ZZ \oplus I)$ by Excision Theorem \ref{excision}.  So we have $\MSE_3(\ZZ \oplus I) = W_E(\ZZ
\oplus I)$. Hence $\MSE_3(R, I) = \MSE_3(\ZZ \oplus I)$ has a group
structure given by Vaserstein's rule (\cite{Sus-Vas}, Theorem 5.2.a).
\end{proof}

\begin{remark}\label{gs3} Lemma \ref{ris} and Theorem \ref{witt relative} show
that $\MSE_3(A, I)$ has an abelian Witt group structure (in particular, 
satisfies $Vv$ rule) whenever
$A$ is a non-singular affine algebra of dimension $3$ and $I$ a principal ideal
in $A$.
\end{remark}

\section{A nice group structure on $\Um_{d}(A)/\E_{d}(A)$}  

We first recall the following result in \cite{vdK-Rao}.
\begin{theorem}\label{rao1}{\rm(\cite{vdK-Rao} Corollary
3.5)} Let $A$ be a regular affine algebra of Krull dimension $3$ over
a $C_1$ field $k$ which is perfect  if its characteristic is $2$ or
$3$. Then the Vaserstein Symbol $V: \Um_3(A)/ \E_3(A) \longrightarrow
W_E(A)$ is an isomorphism.  
\end{theorem}

\begin{remark}\label{igs} 
Let $A$ be a non-singular affine algebra of
dimension $d, d \geq 3$. When $d=3$, above Theorem together with
$($\cite{Sus-Vas}, Theorem 5.2.a$) $ says that $\Um_3(A)/ \E_3(A)$ has a
group structure given by $Vv$ rule. Theorem \ref{wms} says so when $d
\geq 4$. Therefore $\MSE_{d}(A)$ has a group operation $\ast$ defined on
it given by $Vv$ rule whenever $d \geq 3$.
\end{remark}

\begin{lemma}\label{rao}{\rm(\cite{vdK-Rao} Theorem 5.1)}

Let $A$ be a smooth affine algebra of dimension $d \geq 3$ over a
perfect $C_1$ field. Let $1 \leq k \leq d$. Let $v = (v_1, v_2,
\ldots, v_d) \in \Um_d(A)$ and let $T$ be a $k \times k$ matrix over
$A$ with first row $u = (u_1, u_2, \ldots, u_k)$ such that
$\overline{det(T)}$ is a square of a unit in $A/(v_{k+1}, \ldots,
v_d)$. Then $$[(v_1, v_2, \ldots, v_k)\cdot T, v_{k+1}, \ldots, v_d] =
[v] + [u,v_{k+1}, \ldots, v_d].$$ In particular, taking $k=d$, we have
$[v\cdot g] = [v] + [e_1g]$, for $g \in SL_d(A)$.

\end{lemma}

\begin{theorem}\label{fasel}{\rm(\cite{Fasel1} Lemma 3.3)}
Let $S$ be a smooth affine surface over an algebraically closed field
of characteristic different from $2, 3$. Then we have $\SL_2(S) \cap
\E(S) = \SL_2(S) \cap \E_3(S) = \SL_2(S) \cap \ESp_4(S) = \SL_2(S)
\cap \ESp(S)$.
\end{theorem}

We shall need the following relative singular version of J. Fasel's observation 
which can be deduced via (\cite{Fasel}, Corollary $4.2$ and Corollary $5.3$).

\begin{lemma} $($\cite{fsr2}$)$ \label{sp4}
Let $A$ be an affine threefold over an algebraically 
closed field. Let $I$ be an ideal of $A$. Then $\Um_4(A, I) = e_1\Sp_4(A, I)$.
\end{lemma}

\begin{lemma}\label{lift}  Let $A$ be an affine algebra of dimension
$3$ over an algebraically closed field $k$ of characteristic different
from $2, 3$ and let $a \in A$ be such that  $A/(a)$ is smooth and
${\rm dim}(A/(a)) = 2$. Assume that $\Um_4(A, (a)) = e_1\Sp_4(A,
(a))$. If $\overline{\sigma} \in \SL_2(A/(a)) \cap \E_3(A/(a))$ then
it has a lift $\sigma \in \SL_2(A)$. 
\end{lemma}

\begin{proof} The argument is similar to that in (\cite{Sus1}, Lemma $2.1$), 
(also  
see (\cite{Sus-Vas}, Chapter III)). We recall it for the convenience of the 
reader. 

By Theorem \ref{fasel} $\overline{\sigma} \in \SL_2(A/(a))
\cap \E_3(A/(a)) = \SL_2(A/(a)) \cap \ESp_4(A/(a))$. Therefore $\tau
\in \ESp_4(A))$ such that $\overline{\tau} = \overline{\sigma} \perp
I_2$. Note that $e_4 \tau = e_4 \vpmod a$. So by our assumption we
have $\delta \in \Sp_4(A, a)$ so that $e_4\tau = e_4\delta$. Let
$\varepsilon = \tau \delta^{-1} \in \Sp_4(A)$. Then $e_4\varepsilon =
e_4$ and $\overline{\varepsilon} = \overline{\sigma} \perp I_2$. It is
easy to see that $\varepsilon$ will look like  $\begin{pmatrix} \sigma
& 0 & *\\ *    & 1 & *\\ 0    & 0 & 1
\end{pmatrix}$ for some $\sigma \in \SL_2(A)$. Then $\sigma$ is a lift
of $\overline {\sigma}$.
\end{proof} 

\begin{theorem}\label{affine-three-fold}  Let $A$ be a smooth affine
algebra of dimension $3$ over an algebraically closed field $k$ of
characteristic not equal to $2, 3$. Then the group structure on the
orbit space $\MSE_3(A)$ is nice. 
\end{theorem}
\begin{proof} A group structure on $\MSE_3(A)$ exists by Theorem
\ref{rao1}. Let $[v] = [(a, a_1, a_2)]$ and $[w] = [(b, a_1,
a_2)]$. We will show that  $[v]\ast [w] = [(ab, a_1, a_2)].$  Applying
Swan's version of Bertini's Theorem as in (\cite{Swan}), we can add a
general linear combination of $ab, a_1$ to $a_2$ changing it to $a_2'$
and assume that $A/(a_2')$ is a smooth affine surface. Theorem
\ref{rao1} shows that $\MSE_3(A)$ has a group structure {\it
viz.}
\begin{eqnarray*} [(b, a_1, a_2)]\ast [(a, a_1, a_2)] \\ &=&[(b, a_1,
a_2')]\ast [(a, a_1, a_2')]\\ &=&[(a(b + p) - 1, (b + p){a_1}, a_2')]
 \: \hskip1cm \ldots \ldots \hfill  (1).
\end{eqnarray*} where $p$ is chosen so that $ap - 1$ belongs to the
ideal generated by $a_1, a_2'$.   Let `overline' denote the image in
$\overline{A}:= A/(a_2')$ and ${\rm ms}$ the universal Mennicke
symbol. Then, we have 
\begin{align*}  ~~{\rm ms}(\overline{a}(\overline{b}+\overline{p}) -
1, (\overline{b}+\overline{p})\overline{a_1})\\  &=  
{\rm
ms}(\overline{a}(\overline{b}+\overline{p}) - 1,
(\overline{b}+\overline{p})){\rm
ms}(\overline{a}(\overline{b}+\overline{p}) - 1, \overline{a_1}) \\ &=
{\rm ms}(\overline{a}(\overline{b}+\overline{p}) -1, \overline{a_1})\\
&=  {\rm ms}(\overline{a}\overline{b}, \overline{a_1}). 
\end{align*}
So  there exists $\overline{\sigma} \in
\SL_2(\overline{A}) \cap \E_3(\overline{A})$ such that  $
(\overline{a}(\overline{b}+\overline{p}) - 1,
(\overline{b}+\overline{p})\overline{a_1})\overline{\sigma} =
(\overline{a}\overline{b}, \overline{a_1})$. By combining 
Lemma \ref{sp4} and Lemma \ref{lift}, one knows that 
$\overline{\sigma}$  has a lift $\sigma \in \SL_2(A)$. So we have 

\begin{eqnarray*}  [(ab, a_1, a_2)] \\ &=& [(ab, a_1, a_2')]\\ &=&
[(a(b+p)-1, (b+p)a_1, a_2') \begin{pmatrix}\sigma & 0\\ 0 &      1
                               \end{pmatrix}]\\ &=& [(a(b + p) - 1, (b
+ p){a_1}, a_2')]\ast [e_1\begin{pmatrix} \sigma & 0 \\ 0 & 1
                                            \end{pmatrix}] ; \;
\text{by Lemma \ref{rao}} \\ &=& [(a(b + p) - 1, (b + p){a_1},
a_2')]  \; \hskip1cm \ldots \ldots \hfill (2).
\end{eqnarray*} Hence, the result follows from equations (1) and (2).
\end{proof}

\begin{theorem}\label{main nice}  Let $A$ be a smooth affine algebra
over an algebraically closed field $k$ of characteristic not equal to
$2, 3$.  Then, the group structure on the orbit space $\MSE_d(A)$, $d
\geq 3$ is nice, i.e.  it is given by the `coordinate-wise
multiplication' formula: 
$$[(a, a_1, a_2, \ldots, a_{d-1})]\ast [(b, a_1, a_2, \ldots, a_{d-1})] = [(ab, a_1, a_2, \ldots, a_{d-1})].$$ 
\end{theorem}
\begin{proof} We shall proceed by induction on $d$. $d = 3$ is proved
in Theorem \ref{affine-three-fold}. Let $v= (a, a_1, a_2, \ldots,
a_{d-1})$, $w=(b, a_1, a_2, \ldots, a_{d-1})$ be such that $v, w \in
\Um_d(A)$. Applying R. G. Swan's version of Bertini's Theorem as in
(\cite{Swan}), we can add a general linear combination of $ab, a_1,
a_2 \ldots, a_{d-2}$ to $a_{d-1}$ changing it to $a_{d-1}'$ and assume
that $A/(a_{d-1}')$ is a smooth affine algebra of dimension $d-1$. Now
we have a group structure on $\MSE_d(A)$ by Remark \ref{igs}.  Choose
$p \in A$ such that $ap \equiv 1 \vpmod {(a_1, a_2, \ldots
a_{d-1}')}$. Then 
\begin{eqnarray*} [w]\ast [v]&=&[(b, a_1, a_2, \ldots, a_{d-1}')]\ast [(a,
a_1, a_2, \ldots, a_{d-1}')]\\ &=&[(a(b+p) - 1, a_1(b+p), a_2, \ldots,
a_{d-1}')]\, \ldots [1].
\end{eqnarray*}  Going modulo $a'_{d-1}$, we have
\begin{eqnarray*} [(\overline{a}(\overline{b}+\overline{p}) -
\overline{1}, \overline{a_1}(\overline{b}+\overline{p}),
\overline{a_2}, \ldots, \overline{a_{d-2}})]\\ =&&[(\overline{a},
\overline{a_1}, \overline{a_2}, \ldots,
\overline{a_{d-2}})]\ast [(\overline{b}, \overline{a_1}, \overline{a_2},
\ldots, \overline{a_{d-2}})] \quad \text{by Remark \ref{igs}} \\
=&&[(\overline{ab}, \overline{a_1}, \overline{a_2}, \ldots,
\overline{a_{d-2}})] \quad \text{by induction hypothesis.}
\end{eqnarray*} Therefore, 
\begin{eqnarray*} [(a(b+p) - 1, a_1(b+p), a_2, \ldots, a_{d-1}')]  &=&
[(ab, a_1, a_2, \ldots, a_{d-1}')]\\ &=& [(ab, a_1, a_2, \ldots,
a_{d-1})]\, \ldots [2].
\end{eqnarray*} and the result follows from [1] and [2].
\end{proof}
\begin{corollary}
$\MSE_d(A) = \WMS_d(A) = \MS_d(A)$ for $A$ satisfying properties as in earlier theorem.
\end{corollary}

\section{A relative Mennicke--Newman Lemma} 

We begin by recalling the
two cases of the Mennicke--Newman lemma proved  by W. van der Kallen
(following (\cite{Sus}, Lemma 1.2) and (\cite{bms}, Lemma 2.4)).  We then
proceed to prove an analogue of it. 

\medskip

First the relative case:

\begin{lemma} $($\cite{vdK1}, Lemma $3.4$$)$ Let $R$ be a commutative ring
of Krull dimension $d$. Let $v, w$ be  unimodular rows of length $d +
1$ relative to an ideal $I$ of $R$.  Then there exist $\varepsilon_1,
\varepsilon_2 \in \E_{d+1}(R, I)$ such  that $v \varepsilon_1 = (x,
a_2,. \ldots, a_{d+1})$,  $w \varepsilon_2 = (y, a_2,. \ldots,
a_{d+1})$, with  $V(a_2, \ldots, a_{d+1})$ is a union of the closed 
set $V(I + a_2R +
\ldots + a_{d+1}R)$  and finitely many subsets of dimension $0$. 
\end{lemma}

\medskip

Next the absolute case:

\begin{lemma}\label{mn} $($\cite{vdK3}, Lemma $3.2$ $)$ Let $R$ be a ring of stable
dimension $d \leq 2n -3$. Let $v, w \in Um_n(R)$.  Then there are
$\varepsilon_1, \varepsilon_2 \in \E_n(R)$, and $x, y, a_i \in R  $, with
$x + y = 1$ such that $v\varepsilon_1 = (x, a_2, \ldots, a_n)$,
$w\varepsilon_2 = (y, a_2, \ldots, a_n)$. 
\end{lemma}

\medskip

The following are  relative versions  of (\cite{vdK3}, Lemma $3.2$).

\begin{lemma}\label{rm-newman1}{{\rm(}Relative Mennicke--Newman{\rm)} }
Let $R$ be a ring  of stable dimension $d$ with $d \leq 2n - 3$ and
$I$ an ideal in $R$.  Let $v, w \in \Um_n(R, I)$. Then there exists
$\varepsilon_1, \varepsilon_2  \in \E_n(R, I)$ such that
$v\varepsilon_1 = (a_1, a_2, \ldots, a_{n-1}, a)$ and  $w\varepsilon_2
= (a_1, a_2, \ldots, a_{n-1}, b)$ such that $a+b$ is a unit modulo
$(a_1, a_2, \ldots, a_{n-1})$.  
\end{lemma}

\begin{proof} Let $v = (a_1, a_2, \ldots, a_n)$, $w = ( b_1, b_2,
\ldots, b_n) \in \Um_n(R, I)$ . Then $(a_1,  a_2,  a_3 \ldots,
a_{n-1}, b_1, b_2, \ldots,\\ b_{n-1},a_nb_n) \in \Um_{2n-1}(R)$. Since
$Sr(R) = 1 + Sd(R) \leq 2n-2$, we can find $c_1, c_2, \ldots, c_{n-1},
d_1, d_2, \ldots, \\ d_{n-1} \in R$ such that $(a_1 + c_1a_nb_n, a_2 +
c_2a_nb_n, \ldots, a_{n-1} + c_{n-1}a_nb_n, b_1 + d_1a_nb_n, b_2 +
d_2a_nb_n, \ldots, b_{n-1} + \\ d_{n-1} a_nb_n) \in
\Um_{2n-2}(R,I)$. We add multiples of $a_n, b_n$ {\it viz.}
$c_ia_nb_n$ to $a_i$ and $d_ia_nb_n$ to $b_i$, $1 \leq i \leq n -1$ to
assume that the the ideals $(a_1, a_2, \ldots,a_{n-1})$ and $(b_1,
b_2, \ldots, b_{n-1})$ are comaximal. Now adding a suitable $I$
linear combination of $a_1, a_2, \ldots a_{n-1}$ to $a_n$ and that of
$b_1, b_2, \ldots b_{n-1}$ to $b_n$ we can make $ b_n - a_n = a_1 -
b_1$. Therefore adding last coordinates to the first coordinates we
may assume that $a_1 = b_1$. We can do this by $\E_n(A, I)$ action since  for any $u = (1 +i_1, i_2, \ldots, i_n) \in \Um_n(R, I)$ we have 
\[ (1 + i_1 + i_n, i_2, \ldots, i_n) = u E_{n1}(1) =u E_{12}(i_n)E_{n2}(-i_1)E_{21}(1)E_{12}(-i_n)E_{n2}(i_1 + i_2 + i_n)E_{21}(-1). \]


Let $a_1 = b_1 = 1-\lambda, \lambda \in I$.  Now by elementary action
we change $v$ and $w$ to $v_1= (a_1, \lambda^2a_2, \ldots,
\lambda^2a_n)$ and  $w_1 = ( b_1, \lambda^2b_2, \ldots, \lambda^2b_n)$
respectively. Considering the row $(a_1,  \lambda^2a_2,  \lambda^2a_3
\ldots, \lambda^2a_{n-1}, b_1, \lambda^2b_2, \lambda^2b_3 \\, \ldots,
\lambda^2b_{n-1}, \lambda^4a_nb_n) \in \Um_{2n-1}(R)$ and arguing as
in the previous paragraph we change $v_1$, $w_1$  to $v_2 =  (a'_1,
a'_2, \ldots, a'_n)$ and $w_2 = ( b'_1, b'_2, \ldots, b'_n)$
respectively by adding multiples of $\lambda^4a_nb_n$  to the first
$n-1$ coordinates such that the ideals $(a'_1, a'_2,
\ldots,a'_{n-1})$ and $(b'_1, b'_2, \ldots , b'_{n-1})$ are
comaximal.  We have $a'_i \equiv b'_i \vpmod {\lambda^2}, i = 1, 2,
\ldots ,n-1$. Now adding suitable $I$ linear combinations of the first
$n-1$ coordinates to the last we may assume that $a'_n + b'_n =
\lambda$. We note that $a'_i - b'_i =  c_i\lambda^2 = c_i\lambda(a'_n
+ b'_n),  c_i \in R, i = 1, 2, \ldots ,n-1$. Therefore we can add
suitable $\lambda$ multiples of the last coordinate to the first
$n-1$ coordinates to have $a'_i = b'_i$ for $i = 1, 2, \ldots,
n-1$. Since $a'_1 \equiv 1 \vpmod \lambda$, $a'_n + b'_n$ is a unit
modulo the ideal generated by the first $n-1$ coordinates. 
\end{proof} 

\begin{lemma}\label{rm-newman2}{{\rm(}Relative Mennicke--Newman{\rm)} }
Let $R$ be a commutative noetherian ring  and $I$ an ideal in $R$ such that the 
$max(R)$ is a disjoin union of $V(I)$ and finitely many irreducible
closed sets $V_i$ each a noetherian topological space of dimension at
most $d$. Assume $d \leq 2n - 3$ and $n \geq 3$. Let $v, w \in \Um_n(R,
I)$. Then there exists $\varepsilon_1, \varepsilon_2  \in \E_n(R, I)$
such that $v\varepsilon_1 = (a, a_2, a_3, \ldots, a_{n})$ and
$w\varepsilon_2 = (b, a_2, a_3, \ldots, a_{n})$ .
\end{lemma}

\begin{proof}
Note that if $d =0, 1$ then any unimodular row of length atleast $3$ is elementarily completable. So the result follows obviously. Therefore we shall assume that $2 \leq d \leq 2n - 3$. By Remark \ref{vr} stable dimension of $\ZZ \oplus I$ is atmost $d$. We choose  $\tilde{v}, \tilde{w} \in 
\Um_n(\ZZ \oplus I, 0 \oplus I)$  such that $[\tilde{v}], [\tilde{w}] \in 
\MSE_n(\ZZ \oplus I, I)$ correspond to $[v], [w] \in \MSE_n(R, I)$ by $F$ 
(see Theorem \ref{excision}) respectively. Now by Lemma \ref{mn} we have 
$\varepsilon_1, \varepsilon_2 \in \E_n(\ZZ \oplus I)$ such that  
$\tilde{v}\varepsilon _1 = (\tilde{a}, \tilde{a_2}, \ldots, \tilde{a_n})$ and 
$\tilde{w}\varepsilon_2 = (\tilde{b}, \tilde{a_2}, \ldots, \tilde{a_n}), 
\tilde{a} + \tilde{b} = 1$.
Assume $\tilde{a_i} \equiv n_i \vpmod I$ for $i \geq 2$ and $c = g.c.d(n_i)$. 
Then by further elementary action (infact by $\E_n(\ZZ)$ action) we may assume 
that $(\tilde{a_2},\tilde{a_3}, \ldots, \tilde{a_n}) \equiv 
(c, 0, \ldots,0) \vpmod I$. Adding $\tilde{a}\tilde{b}$ to the third 
coordinate we have $(\tilde{a_2},\tilde{a_3}, \ldots, \tilde{a_n}) \equiv 
(c, e, \ldots,0) \vpmod I$ such that $g.c.d(c, e) =1$. Now we shall add 
suitable $\ZZ$ linear combination of $\tilde{a_2}, \tilde{a_3}$ to 
$\tilde{a_1}$ to have $\tilde{a} \equiv \tilde{b} \equiv 1 \vpmod I$. Then we 
shall add a suitable $\ZZ$ multiples $\tilde{a}\tilde{b}$  to the rest and 
have $(\tilde{a}, \tilde{a_2}, \ldots, \tilde{a_n}) \equiv (\tilde{b}, 
\tilde{a_2}, \ldots, \tilde{a_n}) \equiv e_1 \vpmod I$.

Thus we have $\varepsilon'_1, \varepsilon'_2 \in \E_n(\ZZ \oplus I)$ such that $\tilde{v}\varepsilon'_1 = \tilde{v}' = (\tilde{a}, \tilde{a_2}, \ldots, \tilde{a_n}), \tilde{w}\varepsilon'_2 = \tilde{w}' = (\tilde{b}, \tilde{a_2}, \ldots, \tilde{a_n})$. Note that $\tilde{v}, \tilde{v'}, \tilde{w}, \tilde{w'} \in 
\Um_n(\ZZ \oplus I,  0 \oplus I)$. So by Excision Theorem \ref{excision} we have $\varepsilon''_1, \varepsilon''_2 \in \E_n(\ZZ \oplus I,  0 \oplus I)$ such that 
$\tilde{v}\varepsilon''_1 = \tilde{v}',  \tilde{w}\varepsilon''_2 = \tilde{w}'$. Now the result follows taking projection onto $R$ under 
$f: \ZZ \oplus I \rightarrow R$ defined by $f(n, i) = n + i$.
\end{proof}


\section{Niceness of relative orbit space group}  In this section we
shall first establish the relative analogue of results in the previous
section. By Remark \ref{vr} and \ref{gs3} we know that $\MSE_d(A, I), d \geq 3$
has a group structure given by $Vv$ rule whenever $A$ is a nonsingular
affine algebra.  

\begin{theorem}\label{relative-affine-three-fold}  Let $A$ be a smooth
affine algebra of dimension $3$ over an algebraically closed field $k$
of characteristic not equal to $2, 3$ and $I$ a principal ideal in $A$. Then the group structure on the orbit space
$\Um_3(A, I)/\E_3(A, I)$ is nice. 
\end{theorem}

\begin{proof} By Remark \ref{gs3} we already have a group structure on
$\MSE_3(A, I)$ given by $Vv$ rule  . Let $[v] = [(a_1, a_2, a)]$, $[w]
= [(a_1, a_2, b)]$, $a_1 = 1- \lambda, \lambda \in I$. We shall show that $[w]\ast [v] =
[(a_1, a_2, ab)]$. Note that $(a_1, \lambda a_2, \lambda ab) \in
\Um_3(A, I)$ Applying Swan's version of Bertini's Theorem in
(\cite{Swan}), we can add a general linear combination of $\lambda ab,
\lambda a_2$ to $a_1$ changing it to $a_1'$ and assume that $A/(a_1')$
is a smooth affine surface. Let $a_1' = 1 - \eta, 
\eta \in I$. Now by $Vv$ rule we
have
$$[w]\ast [v] = [(a_1', a_2, b)]\ast [(a_1', a_2, a)] =  
[(a_1', a_2(b+p), a(b+p) - \eta)].$$
where $(b_1, b_2, p) \in \Um_3(A, I)$  such that $a_1b_1+ a_2b_2 + ap
= 1$. Now in  $\overline{A}:= A/(a_1')$ we have $\overline{\eta} =
1$. Arguing as in Theorem \ref{affine-three-fold} we have
$\overline{\sigma} \in \SL_2(\overline{A}) \cap \E_3(\overline{A})$
such that  $ ( \overline{a_2}(\overline{b}+\overline{p}),
\overline{a}(\overline{b}+\overline{p}) -
\overline{\eta})\overline{\sigma} = ( \overline{a_2},
\overline{a}\overline{b})$. By Lemma \ref{lift}, $\overline{\sigma}$
has a lift $\sigma \in \SL_2(A)$. Clearly  $(a_2(b+p), a(b+p) -
\eta)\sigma \equiv (a_2, ab) \vpmod {Ia_1'}$. So 
\begin{eqnarray*} [(a_1', a_2, ab)]  &=& [(a_1', a_2(b+p), a(b+p) -
\eta) \begin{pmatrix}1 & 0\\ 0 & \sigma
                                                                    \end{pmatrix}]\\
&=& [(a_1', a_2(b+p), a(b+p) - \eta) \begin{pmatrix}1 & 0\\ 0 & \eta
\sigma
                                                                    \end{pmatrix}]\\
&=& [(a_1', a_2(b+p), a(b+p) - \eta)]\ast [(a_1', e_1 \eta \sigma)] ; \;
\text{by $Vv$ rule} \\ &=& [(a_1', a_2(b+p), a(b+p) - \eta)].
\end{eqnarray*} The last equality is true since $(a_1', e_1 \eta
\sigma) = e_1(1\perp \sigma)^{-1}E_{21}(1)E_{12}(\eta)E_{2
1}(-1)(1\perp \sigma) \in e_1\E_3(A, I)$.  Therefore we have 
\begin{eqnarray*} [(a_1, a_2, b)]\ast [(a_1, a_2, a)] &=& [(a_1', a_2,
b)]\ast [(a_1', a_2, a)]\\ &=&[(a_1', a_2(b+p), a(b+p) - \eta)]\\ &=&
[(a_1', a_2, ab)]\\ &=&[(a_1, a_2, ab)].
\end{eqnarray*}
\end{proof}

\begin{theorem}\label{main relative nice}  Let $A$ be a smooth affine
algebra over an algebraically closed field $k$ of characteristic not
equal to $2, 3$ of dimension $d \geq 4$ and $I$ an ideal in $A$. Then,
the group structure on the orbit space $\MSE_d(A, I)$ is
nice, i.e.  it is given by the `coordinate-wise  multiplication'
formula: 
$$[(a_1, a_2, \ldots, a_{d-1}, a)]\ast [(a_1, a_2, \ldots, a_{d-1}, b)] = 
[(a_1, a_2, \ldots, a_{d-1}, ab)].$$ 
\end{theorem}

\begin{proof} Let $v= (a_1,a_2, \ldots, a_{d-1}, a)$, $w=(a_1, a_2, \ldots, a_{d-1}, b)$ be such
that $v, w \in \Um_d(A, I)$, $a_1 = 1 - \lambda$. Then  $(a_1,\lambda
a_2, \ldots, \lambda a_{d-1}, \lambda ab) \in  \Um_n(A, I)$.  Applying
R. G. Swan's version of Bertini's Theorem in (\cite{Swan}) we can add a
general linear combination of $\lambda ab, \lambda a_{d-1}, \lambda
a_{d-2} \ldots, \lambda a_{2}$ to $a_1$ changing it to $a_1'$ and
assume that $A/(a_{1}')$ is a smooth affine algebra of dimension
$d-1$. Let $a_1' = 1 - \eta$. Note that in $\overline{A} = A / a_1'$
we have $\overline{\eta} = 1$. We already have a group structure on
$\MSE_{d}(A, I), d \geq 4$ by Remark \ref{vr} 
.  We choose $(b_1,
b_2, \ldots, b_{d-1}, p)\in \Um_d(A, I)$ such that $a_1'b_1 +
\Sigma_{i=2}^{d-1}a_ib_i + ap = 1$. Then 
\begin{eqnarray*} [w]\ast [v]& =& [(a_1', a_2, \ldots, a_{d-1},
b)]*[(a_1', a_2, \ldots, a_{d-1}, a)]\\ &=& [(a_1', a_2, \ldots,
a_{d-2}, a_{d-1}(b+p), a(b+p) - \eta))].
\end{eqnarray*} Going modulo $a_{1}'$, we have
\begin{eqnarray*} [( \overline{a_2}, \ldots, \overline{a_{d-2}},
\overline{a_{d-1}}(\overline{b}+\overline{p}),
\overline{a}(\overline{b}+\overline{p}) - \overline{\eta}))] =&&[(
\overline{a_2}, \ldots, \overline{a_{d-1}}, \overline{b})] \ast  [(
\overline{a_2}, \ldots, \overline{a_{d-1}}, \overline{a})] \quad
\text{by Remark \ref{igs}} \\ =&&[(\overline{a_2}, \overline{a_3},
\ldots, \overline{a_{d-1}}, \overline{ab}, )] \quad \text{by Theorems
 \ref{affine-three-fold}, \ref{main nice}.}
\end{eqnarray*} Therefore we have  $(a_2, \ldots, a_{d-2},
a_{d-1}(b+p), a(b+p) - \eta))\alpha \equiv (a_2, \ldots, a_{d-1},
ab)  \vpmod{Ia_1'}$ for some $\alpha \in \E_{d-1}(I)$. So 
\begin{eqnarray*} [(a_1', a_2, \ldots, a_{d-2}, a_{d-1}(b+p), a(b+p) -
\eta))]&=&[(a_1', a_2, \ldots, a_{d-1}, ab)] \\ &=& [(a_1, a_2,
\ldots, a_{d-1}, ab)]; \, \text{in} \;  \MSE_d(A, I).
\end{eqnarray*}
\end{proof}

As a consequence one can deduce as in Corollary \ref{cor:completable-affine-d}:

\begin{corollary}\label{cor:completable-affine-dd}  Let $A$ be a smooth
affine algebra of dimension $d \geq 3$ over an algebraically closed
field $k$, with ${\rm char}~k \neq 2, 3$. $I$ satisfies conditions given in  Theorems \ref{relative-affine-three-fold} and \ref{main relative nice} depending on $d = 3$ or $d > 3$ respectively.  Let $\sigma \in \SL_{d}(A, I)
\cap \E_{d+1}(A, I)$. Then, $[e_1\sigma] = [e_1]$ in $\MSE_{d}(A,
I)$ i.e. $e_1\sigma$ is relatively elementarily completable to $e_1$.  
\end{corollary}

\begin{corollary}
$\MSE_d(A, I) = \WMS_d(A, I) = \MS_d(A, I)$ where $A$ and $I$ satisfy conditions given in  Theorems \ref{relative-affine-three-fold} and \ref{main relative nice} depending on $d = 3$ or $d > 3$ respectively.
\end{corollary}

\begin{question} Let $R$ be a local ring of Krull dimension $d\geq 2$. Let $I$ 
be an ideal of $R$. Is the group structure on $\MSE_d(R[X], I[X])$ nice? 
\end{question}

\begin{acknowledgement}

The first author is grateful to SPM Fellowship,CSIR for financial support.
\end{acknowledgement}

\medskip

\noindent
{\it Email: Anjan Gupta $<$anjan@math.tifr.res.in$>$}\\
{\it Email: Anuradha Garge $<$anuradha@cbs.ac.in$>$ \\
{\it Email: Ravi A. Rao $<$ravi@math.tifr.res.in$>$}

\end{document}